\documentclass[twoside,leqno,11pt]{article}

\usepackage{amssymb}
\usepackage{amsmath}
\usepackage{amsfonts}

\newtheorem{theorem}{Theorem}

\newtheorem{lemma}[theorem]{Lemma}

\newtheorem{proposition}[theorem]{Proposition}
\newtheorem{remark}[theorem]{Remark}
\newenvironment{proof}[1][Proof]{\noindent\textbf{#1.} }{\ \rule{0.5em}{0.5em}}

\def\rank{\mathop{\mathrm{rank}}\nolimits}
\def\O{\mathcal{O}}
\def\OP{\mathcal{OP}}
\def\T{\mathcal{T}}
\def\S{\mathcal{S}}
\def\E{\mathcal{E}}

\def\R{\mathbb R}

\title{On relative ranks of finite transformation semigroups with restricted range}

\author{Ilinka Dimitrova \\
\textit{Faculty of Mathematics and Natural Science}\\
\textit{South-West University "Neofit Rilski"}\\
\textit{2700 Blagoevgrad, Bulgaria}\\
\textit{e-mail: ilinka\_dimitrova@swu.bg}\\
~~\\
J\"{o}rg Koppitz \\
\textit{Institute of Mathematics and Informatics}\\
\textit{Bulgarian Academy of Sciences}\\
\textit{1113 Sofia, Bulgaria}\\
\textit{e-mail: koppitz@math.bas.bg}}

\begin{document}

\maketitle

\begin{abstract}
In this paper, we determine the relative rank of the semigroup
$\T(X,Y)$ of all transformations on a finite chain $X$ with restricted range $Y\subseteq X$ modulo the set $\OP(X,Y)$ of all orientation-preserving transformation in $\T(X,Y)$. Moreover, we state the relative rank of the semigroup $\OP(X,Y)$ modulo the set $\O(X,Y)$ of all order-preserving transformations in $\OP(X,Y)$. In both cases we characterize the minimal relative generating sets.
\end{abstract}


\\

\section{Introduction and Preliminaries}

Let $S$ be a semigroup. The \textit{rank} of $S$ (denoted by $\rank S$) is defined to be the minimal number of elements of a generating set of $S$. The ranks of various well known semigroups have been calculated \cite{GH, GH1, Howie, HMcF}.
For a set $A \subseteq S$, the \textit{relative rank} of $S$ modulo $A$, denoted by $\rank (S : A)$, is the minimal cardinality of a set $B\subseteq S$ such that $A \cup B$ generates $S$. It follows immediately from the definition that $\rank (S : \emptyset) = \rank S$, $\rank ( S : S) = 0$, $\rank (S : A ) = \rank (S : \langle A \rangle )$ and $\rank (S : A) = 0$ if and only if $A$ is a generating set for $S$.
The relative rank of a semigroup modulo a suitable set was first introduced by Ru\v skuc in \cite{Ruskuc} in order to describe the generating sets of semigroups with infinite rank. In \cite{HRH}, Howie, Ru\v skuc, and Higgins  considered the relative ranks of the monoid $\T(X)$ of all full transformations on $X$, where $X$ is an infinite set, modulo some distinguished subsets of $\T(X)$.
They showed that $\rank (\T(X) : \S(X)) = 2$, $\rank (\T(X) : \E(X)) = 2$ and $\rank (\T(X) : J) = 0$, where $\S(X)$ is the symmetric group on $X$, $\E(X)$ is the set of all idempotent transformations on $X$ and $J$ is the top $\mathcal{J}$-class of $\T(X)$, i.e.
$J=\{\alpha\in\T(X) : |X\alpha|=|X|\}$.
But also if the rank is finite, the relative rank gives information about the generating sets. In the present paper, we
will determine the relative rank for a particular semigroup of
transformations on a finite set.

Let $X$ be a finite chain, say $X=\{1<2<\cdots <n\}$, for a natural number $n$.
A transformation $\alpha \in \T(X)$ is called \textit{order-preserving} if $x\le y$ implies $x\alpha\le y\alpha$, for all $x,y\in X$.
We denote by $\O(X)$ the submonoid of $\T(X)$ of all order-preserving full transformations on $X$.
The relative rank of $\T(X)$ modulo $\O(X)$ was considered by Higgins, Mitchell, and Ru\v{s}kuc in \cite{HMR}.
They showed that $\rank(\T(X) : \O(X)) = 1$, when $X$ is an arbitrary countable chain or an arbitrary well-ordered set,
while $\rank(\T(\R) : \O(\R))$ is uncountable, by considering the usual order of the set $\R$ of real numbers.
In \cite{DFK}, Dimitrova, Fernandes, and Koppitz studied the relative rank of the semigroup $\O(X)$ modulo $J=\{\alpha\in\O(X) : |X\alpha|=|X|\}$,
for an infinite countable chain $X$.
We say that a transformation $\alpha \in\T(X)$ is \textit{orientation-preserving} if there
are subsets $X_{1},X_{2}\subseteq X$ with $\emptyset \neq X_{1}<X_{2}$, (i.e.
$x_{1}<x_{2}$ for $x_{1}\in X_{1}$ and $x_{2}\in X_{2}$), $X=X_{1}\cup X_{2}$,
and $x\alpha \leq y\alpha$, whenever either $(x,y)\in X_{1}^{2}\cup X_{2}^{2}$ with $x\leq y$ or $(x,y)\in X_{2}\times X_{1}$.
Note that $X_{2}=\emptyset$ provides $\alpha \in \O(X)$.
We denote by $\OP(X)$ the submonoid of $\T(X)$ of all orientation-preserving full transformations on $X$. An equivalent notion of an orientation-preserving transformation was first introduced by McAlister in \cite{McA} and, independently, by Catarino and Higgins in \cite{CH}. It is clear that $\O(X)$ is a submonoid of $\OP(X)$, i.e. $\O(X)\subset \OP(X)\subset \T(X)$. It is interesting to note that the relative rank of $\T(X)$ modulo $\OP(X)$ as well as the relative rank of $\OP(X)$ modulo $\O(X)$ is one (see \cite{CH, HRH}), but the situation will change if one considers transformations with restricted range.

Let $Y=\{a_{1}<a_{2}<\cdots <a_{m}\}$ be a nonempty subset of $X$, for a natural number $m\leq n$, and denote by $\T(X,Y)$ the subsemigroup $\{\alpha \in \T(X) : X \alpha \subseteq Y\}$ of $\T(X)$ of all transformations with range (image) restricted to $Y$.
The set $\T(X,Y)$ coincides with $\T(X)$, whenever $Y=X$ (i.e. $m=n$).
In 1975, Symons \cite{Symons} introduced and studied the semigroup $\T(X,Y)$, which is called
semigroup of transformations with restricted range.
Recently, the rank of $\T(X,Y)$ was computed by Fernandes and Sanwong
in \cite{FS}. They proved that the rank of $\T(X,Y)$ is the Sterling number $S(n,m)$ of second kind with $\left\vert X\right\vert =n$
and $\left\vert Y\right\vert =m$.
The rank of the order-preserving counterpart $\O(X,Y)$ of $\T(X,Y)$ was studied in \cite{FHQS1} by Fernandes, Honyam, Quinteiro, and Singha.
The authors found that
$\rank\O(X,Y)=\left(
\begin{array}{c}
n-1 \\
m-1%
\end{array}%
\right) +\left\vert Y^{\#}\right\vert,$
where $Y^{\#}$ denotes the set of all
$y\in Y$ with one of the following properties:
(i) $y$ has no successor in $X$;
(ii) $y$ is no successor of any element in $X$;
(iii) both the successor of $Y$ and the element whose successor is $y$ belong to $Y$.
Moreover, the regularity and the rank of the semigroup $\OP(X,Y)$ were studied by the same authors in \cite{FHQS2}.
They showed that
$\rank\OP(X,Y)=\left(
\begin{array}{c}
n \\
m
\end{array}%
\right)$.
In \cite{TK}, Tinpun and Koppitz studied the relative rank of $\T(X,Y)$ modulo $\O(X,Y)$ and proved that
$\rank(\T(X,Y):\O(X,Y))=
S(n,m)-
\left(
\begin{array}{c}
n-1 \\
m-1
\end{array}
\right) + a$,
where $a \in \{0,1\}$ depending on the set $Y$.
In this paper, we determine the relative rank of $\OP(X,Y)$ modulo $\O(X,Y)$ as well as the relative
rank of $\T(X,Y)$ modulo $\OP(X,Y)$.

Let $\alpha \in \T(X,Y)$. The kernel of $\alpha$ is the
equivalence relation $\ker\alpha $ with $(x,y)\in \ker\alpha $ if $x\alpha
=y\alpha$. It corresponds uniquely to a partition on $X$. This justifies to regard $\ker \alpha $ as a partition on $X$. We will call a
block of this partition as $\ker\alpha$-class. In particular, the sets $x\alpha^{-1}=\{y\in X:y\alpha =x\}$, for $x\in X\alpha$, are the $\ker\alpha$-classes.
We say that a partition $P$ is a subpartition of a
partition $Q$ of $X$ if for all $p\in P$ there is $q\in Q$ with
$p\subseteq q$. A set $T\subseteq X$ with $\left\vert T\cap x\alpha^{-1}\right\vert =1$, for all $x\in X\alpha$, is called a transversal of $\ker\alpha$. Let $A\subseteq X$. Then $\alpha |_{A}:A\rightarrow Y$ denotes the restriction of $\alpha$ to $A$ and $A$ will be called convex if $x<y<z$ with $x,z\in A$ implies $y\in A$.

Let $l\in \{1,\ldots ,m\}$. We denote by $\mathcal{P}_{l}$ the set of all
partitions $\{A_{1},\ldots,A_{l}\}$ of $X$ such that $A_{2}<A_{3}<\cdots
<A_{l}$ are convex sets (if $l>1$) and $A_{1}$ is the union of two convex
sets with $1,n\in A_{1}.$
Further, let $\mathcal{Q}_{l}$ be the set of all partitions $\{A_{1},\ldots,A_{l}\}$
of $X$ such that $A_{1}<A_{2}<\cdots <A_{l}$ are convex and let $\mathcal{R}_{l}$
be the set of all partitions of $X$, which not belong to $\mathcal{Q}_{l} \cup \mathcal{P}_{l}$. We observe that $\ker \beta \in \mathcal{Q}_{l} \cup \mathcal{P}_{l}$, whenever $\beta\in \OP(X,Y)$ with $|X\beta|=l$. In particular, $\ker \beta \in \mathcal{Q}_{l}$, whenever $\beta \in \O(X,Y)$.

Let us consider the case $l=m>1$. For $P\in \mathcal{P}_{m}$ with the blocks $A_{1}$,$A_{2}<\cdots <A_{m}$,
let $\alpha _{P}$ be the transformation on $X$ defined by
$$x\alpha_{P}:=a_{i}, \text{ whenever } x\in A_{i} \text{ for } 1\leq i\leq m$$
in the case $1\notin Y$ or $n\notin Y$ and
$$x\alpha_{P}:=\left\{\begin{array}{ll}
                 a_{i+1}, & \text{if } x\in A_{i} \text{ for } 1\leq i<m \\
                 a_{1} & \text{if } x\in A_{m}
               \end{array}\right.$$
in the case $1,n\in Y$.
Clearly, $\ker \alpha _{P}=P$. With $X_{1}=\{1,\ldots ,\max A_{m}\}$, $X_{2}=\{\max A_{m}+1,\ldots ,n\}$ in the case $1\notin Y$ or $n\notin Y$ and $X_{1}=\{1,\ldots ,\max A_{m-1}\}$, $X_{2}=\{\max A_{m-1}+1,\ldots ,n\}$ in the case $1,n\in Y$, where $\max A_{m}$ ($\max A_{m-1}$) denotes the greatest element in the set $A_{m}$ ($A_{m-1}$, respectively), we can easy verify that $\alpha _{P}$ is orientation-preserving.

Further, let $\eta \in \T(X,Y)$ be defined by
\[
x\eta :=\left\{
\begin{array}{ll}
a_{i+1} & \mbox{if }~~ a_{i}\leq x<a_{i+1}\mbox{ for }1\leq i<m \\
a_{1} & \mbox{if }~~ x=a_{m} \\
a_{\Gamma} & \mbox{otherwise}
\end{array}%
\right. ~~\mbox{ with }~~ \Gamma :=\left\{
\begin{array}{ll}
1 & \mbox{if }~~ 1\notin Y \\
2 & \mbox{otherwise,}
\end{array}
\right.
\]
in the case $1\notin Y$ or $n\notin Y$ and
\[
x\eta :=\left\{
\begin{array}{lll}
a_{i+1} & \text{if} & a_{i}\leq x<a_{i+1},~~ 1\leq i<m \\
a_{1}=1 & \text{if} & x = a_{m}=n
\end{array}%
\right.
\]%
in the case that $1,n\in Y$. Notice that $P_0 :=\ker \eta \in \mathcal{P}_{m}$ if $1\notin Y$ or $n\notin Y$ and $\ker \eta \in \mathcal{Q}_{m}$ if $1,n\in Y$.
In fact, $\eta \in \OP(X,Y)$ with $X_{1}=\{1,2,\ldots,a_{m}-1\}$ and $X_{2}=\{a_{m},a_{m}+1,\ldots,n\}$.
Moreover, $\eta |_{Y}$ is a permutation on $Y$, namely
$$\eta |_{Y}=\left(
\begin{array}{cccc}
a_{1} & \ldots & a_{m-1} & a_{m} \\
a_{2} & \cdots & a_{m} & a_{1}%
\end{array}%
\right).$$
We will denote by $\S(Y)$ the set of all permutations on $Y$.
Note, $\beta \in \O(X,Y)$ implies that either $\beta|_{Y}$ is the identity mapping on $Y$ or $\beta|_{Y} \notin \S(Y)$.

\section{The relative rank of $\OP(X,Y)$ modulo $\O(X,Y)$}

In this section, we determine the relative rank of $\OP(X,Y)$ modulo $\O(X,Y)$. A part of these results were presented at the 47-th Spring Conference of the Union of Bulgarian Mathematicians in March 2018 and are published in the proceedings of this conference \cite{DKT}.

If $m=1$ then $\OP(X,Y)$ is the set of
all constant mappings and coincides with $\O(X,Y)$, i.e. $\rank(\OP(X,Y):\O(X,Y))=0$. So, we admit that $m>1$.

First, we will show that
$$\mathcal{A}:=\{\alpha_{P}: P\in \mathcal{P}_{m}\}\cup \{\eta\}$$
is a relative generating set of $\OP(X,Y)$ modulo $\O(X,Y)$.
Notice that $\eta=\alpha_{P_0}$ if $1\notin Y$ or $n\notin Y$.

\begin{lemma}\label{le1}
For each $\alpha \in \OP(X,Y)$ with $\rank\alpha =m$, there is $\widehat{%
\alpha }\in \{\alpha _{P}: P\in \mathcal{P}_{m}\} \cup \O(X,Y)$ with $\ker \alpha =\ker \widehat{\alpha }$.
\end{lemma}

\begin{proof}
Let $\alpha \in \OP(X,Y)$ and let $X_{1},X_{2}\subseteq X$ as in the
definition of an orientation-preserving transformation. If $X_{2}=\emptyset$
then $\alpha \in \O(X,Y)$. Suppose now that $X_{2}\neq \emptyset $ and let $%
X_{1}\alpha =\{x_{1}<\cdots <x_{r}\}$ and $X_{2}\alpha =\{y_{1}<\cdots
<y_{s}\}$ for suitable natural numbers $r$ and $s$. We observe that $%
X_{1}\alpha $ and $X_{2}\alpha $ have at most one joint element (only $x_{1}=y_{s}$ could be possible). If $x_{1}\neq y_{s}$ then $\ker \alpha =\{x_{1}\alpha
^{-1}<\cdots <x_{r}\alpha ^{-1}<y_{1}\alpha ^{-1}<\cdots <y_{s}\alpha
^{-1}\}=\ker \widehat{\alpha }$ with $\widehat{\alpha }=\left(
\begin{array}{cccccc}
x_{1}\alpha ^{-1} & \cdots & x_{r}\alpha ^{-1} & y_{1}\alpha ^{-1} & \cdots
& y_{s}\alpha ^{-1} \\
a_{1} & \cdots & a_{r} & a_{r+1} & \cdots & a_{r+s}%
\end{array}%
\right) \in \O(X,Y)$. If $x_{1}=y_{s}$ then $1,n\in x_{1}\alpha ^{-1} =
y_{s}\alpha ^{-1}$ and $\ker \alpha = \ker \alpha_P$ with $P = \{x_{1}\alpha ^{-1}, x_{2}\alpha ^{-1}<\cdots <x_{r}\alpha^{-1}<y_{1}\alpha ^{-1}<\cdots <y_{s-1}\alpha ^{-1}\}\in \mathcal{P}_{m}$.
\end{proof}

\begin{proposition}\label{le2}
$\OP(X,Y)=\left\langle \O(X,Y),\mathcal{A}\right\rangle $.
\end{proposition}

\begin{proof}
Let $\beta \in \OP(X,Y)$ with $\rank\beta =m$. Then there is $\theta \in
\{\alpha _{P}:P\in \mathcal{P}_{m}\}\cup \O(X,Y)$ with $\ker \beta =\ker
\theta$ by Lemma \ref{le1}. In particular, there is $r\in \{0,\ldots ,m-1\}$ with $%
a_{1}\theta ^{-1}=a_{r+1}\beta ^{-1}$. Then it is easy to verify that $\beta
=\theta \eta ^{r}$, where $\eta ^{0}=\eta ^{m}$. \newline
Admit now that $i=\rank\beta <m$. Suppose that $\ker \beta \in \mathcal{P}_{i}$, say $\ker \beta =\{A_{1},A_{2}<\cdots <A_{i}\}$
with $1,n\in A_{1}$. Then there is a subpartition $P'\in \mathcal{P}_{m}$ of $\ker\beta$. We put $\theta =\alpha _{P'}$, $a=\min X\beta$, and let $T$ be a
transversal of $\ker \theta$. In particular, we have $Y=\{x(\theta
|_{T})\eta ^{k}:x\in T\}$ for all $k\in \{1,\ldots ,m\}$. Since both
mappings $\theta |_{T}:T\rightarrow Y$ and $\eta |_{Y}:Y\rightarrow Y$ are
bijections, there is $k\in \{1,\ldots ,m\}$ with $a_{1}((\theta |_{T})\eta
^{k})^{-1}\beta =a$ and $a_{1}((\theta |_{T})\eta ^{k+1})^{-1}\beta \neq a$.
Moreover, since $(\theta |_{T})\eta^{k}$ is a bijection from $T$ to $Y$ and both transformations $\theta \eta^{k}$ and $\beta $ are orientation-preserving, it is easy to verify that $f^{\ast }=((\theta |_{T})\eta ^{k})^{-1}\beta $ can be extended to an orientation-preserving transformation $f$ defined by
$$xf=\left\{
\begin{array}{lll}
a_{1}f^{\ast } & \text{if} & x<a_{1}  \\
a_{i}f^{\ast } & \text{if} & a_{i}\leq x<a_{i+1},~~ 1\leq i<m \\
a_{m}f^{\ast } & \text{if} & a_{m}\leq x,
\end{array}%
\right.$$
i.e. $f$ and $f^{\ast }$ coincide on $Y$. Moreover, $%
a_{1}f=a_{1}f^{\ast }=a_{1}((\theta |_{T})\eta ^{k})^{-1}\beta =a$. In order
to show that $f$ is order-preserving, it left to verify that $nf\neq a$.
Assume that $nf=a$, where $n\geq a_{m}$. Then $nf=a_{m}f^{\ast }=a_{m}f$,
i.e. $(n,a_{m})\in \ker f$ and $n\eta =a_{m}\eta =a_{1}$. So, there is $%
x^{\ast }\in T$ such that $x^{\ast }((\theta |_{T})\eta ^{k})=a_{m}$, i.e. $%
x^{\ast }=a_{m}((\theta |_{T})\eta ^{k})^{-1}$. Now, we have $%
a=nf=a_{m}f^\ast=a_{m}((\theta |_{T})\eta ^{k})^{-1}\beta =a_{m}(\eta
^{k}|_Y)^{-1}(\theta |_{T})^{-1}\beta =a_{1}(\eta|_Y)^{-1}(\eta ^{k}|_Y)^{-1}(\theta
|_{T})^{-1}\beta =a_{1}((\theta |_{T})\eta ^{k+1})^{-1}\beta \neq a$, a
contradiction. \\
Finally, we will verify that $\beta =\theta \eta ^{k}f\in
\left\langle \O(X,Y),\mathcal{A}\right\rangle $. For this let $x\in X$. Then there is $%
\widetilde{x}\in T$ such that $(x,\widetilde{x})\in $ $\ker \beta $. So, we
have $x\theta \eta ^{k}f=x\theta \eta ^{k}f^{\ast }=\widetilde{x}\theta \eta
^{k}((\theta |_{T})\eta ^{k})^{-1}\beta =\widetilde{x}\beta =x\beta $.
\newline
Suppose now that $\ker \beta \notin \mathcal{P}_{i}$ and thus, $\ker \beta \in \mathcal{Q}_{i}$.
Let $X\beta =\{b_{1},\ldots b_{i}\}$ such that $b_{1}\beta ^{-1}<\cdots <b_{i}\beta ^{-1}
$. Then we define a transformation $\varphi $ by $x\varphi =a_{j}$ for all $%
x\in b_{j-1}\beta ^{-1}$ and $2\leq j\leq i+1$. Clearly, $\varphi \in \O(X,Y)$%
. Further, we define a transformation $\nu \in \T(X,Y)$ by
$$x\nu =\left\{
\begin{array}{ll}
b_{j-1} & \text{if }  a_{j}\leq x < a_{j+1},~~ 2\leq j\leq i \\
b_{i} & \text{otherwise.}
\end{array}%
\right. $$
Since $\beta $ is orientation-preserving, there is $k\in
\{1,\ldots ,i\}$ such that $k=i$ or $b_{1}<\cdots <b_{k-1}>b_{k}<\cdots
<b_{i}$. Then $X_{1}=\{a_{1},\ldots ,a_{k+1}-1\}$ and $X_{2}=\{a_{k+1},
\ldots ,n\}$ gives a partition of $X$ providing that $\nu $ is
orientation-preserving. Clearly, $\rank\nu =i$ and $1\nu =n\nu =b_{i}$.
Thus, it is easy to verify that $\ker \nu \in \mathcal{P}_{i}$. Hence, $\nu
\in \left\langle \O(X,Y),\mathcal{A}\right\rangle $ by the previous case and it remains
to show that $\beta =\varphi \nu \in \left\langle \O(X,Y),\mathcal{A}\right\rangle$.
For this let $x\in X$. Then $x\in b_{j}\beta ^{-1}$ for some $j\in
\{1,\ldots ,i\}$, i.e. $x\varphi \nu =a_{j+1}\nu=b_{j}=x\beta $.
\end{proof}
~~\\

The previous proposition shows that $\mathcal{A}$ is a relative generating set for $\OP(X,Y)$
modulo $\O(X,Y)$. It remains to show that $\mathcal{A}$ is of minimal size.

\begin{lemma}\label{le3}
Let $B\subseteq \OP(X,Y)$ be a relative generating set of $\OP(X,Y)$ modulo $%
\O(X,Y)$. Then $\mathcal{P}_{m}\subseteq \{\ker \alpha :\alpha \in B\}$.
\end{lemma}

\begin{proof}
Let $P\in \mathcal{P}_{m}$. Since $\alpha _{P}\in \OP(X,Y)=\left\langle
\O(X,Y),B\right\rangle $, there are $\theta _{1}\in \O(X,Y)\cup B$ and $\theta
_{2}\in \OP(X,Y)$ with $\alpha _{P}=\theta _{1}\theta _{2}$. Because of $%
\rank\alpha _{P}=m$, we obtain $\ker \alpha _{P}=\ker \theta _{1}$. Since $%
1\alpha _{P}=n\alpha _{P}$, we conclude that $\theta _{1}\notin \O(X,Y)$,
i.e. $\theta _{1}\in B$ with $\ker \theta _{1}=\ker \alpha _{P}=P$.
\end{proof}
~~\\

In order to find a formula for the number of elements in $\mathcal{P}_{m}$,
we have to compute the number of possible partitions of $X$ into $m+1$ convex sets.
This number is $\left(
\begin{array}{c}
n-1 \\
m%
\end{array}%
\right)$. Thus, we have

\begin{remark}\label{re1}
$\left\vert \mathcal{P}_{m}\right\vert =\left(
\begin{array}{c}
n-1 \\
m
\end{array}
\right)$.
\end{remark}

Now, we are able to state the main result of the section. The relative rank
of $\OP(X,Y)$ modulo $\O(X,Y)$ depends of the fact whether both $1$ and $n$
belong to $Y$ or not.

\begin{theorem}\label{th5} For each $1 < m < n \in \mathbb{N}$,
\begin{enumerate}
  \item $\rank(\OP(X,Y):\O(X,Y))=\left(
\begin{array}{c}
n-1 \\
m%
\end{array}%
\right)$ if $1\notin Y$ or $n\notin Y$;
  \item $\rank(\OP(X,Y):\O(X,Y))=1+\left(
\begin{array}{c}
n-1 \\
m%
\end{array}%
\right)$ if $\{1,n\}\subseteq Y$.
\end{enumerate}
\end{theorem}

\begin{proof}
1. Note that $\ker \eta \in \mathcal{P}_{m}$ and $\eta = \alpha_{P_0}$. Hence, the set $\mathcal{A}=\{\alpha_{P}:P\in \mathcal{P}_{m}\}$ is a generating set of $\OP(X,Y)$ modulo $\O(X,Y)$ by Proposition \ref{le2}, i.e. the relative rank of $\OP(X,Y)$ modulo $\O(X,Y)$ is bounded by the cardinality of $\mathcal{P}_{m}$,
which is $\left(
\begin{array}{c}
n-1 \\
m%
\end{array}%
\right)$ by Remark \ref{re1}. But this number cannot be reduced by Lemma \ref{le3}.\\
\\
2. Let $B\subseteq \OP(X,Y)$ be a relative generating set of $\OP(X,Y)$ modulo $%
\O(X,Y)$. By Lemma \ref{le3}, we know that $\mathcal{P}_{m}\subseteq \{\ker \alpha
:\alpha \in B\}$. Assume that the equality holds. Note that $\ker \eta \in \mathcal{Q}_{m}$
and $\eta$ is not order-preserving. Hence, there are $\theta _{1},\ldots ,\theta _{l}\in
\O(X,Y)\cup B$, for a suitable natural number $l$, such that $\eta
=\theta _{1}\cdots \theta _{l}$. From $\rank\eta =m$, we obtain $\ker \theta
_{1}=\ker \eta $ and $\rank\theta _{i}=m$ for $i\in \{1,\ldots ,l\}$ and thus,
$\{1,n\}\subseteq Y$ implies $(1,n)\notin \ker \theta _{i}$ for $i \in
\{2,\ldots ,l\}$. This implies $\theta _{2},\ldots ,\theta _{l}\in \O(X,Y)$.
Since $\ker \theta _{1}=\ker \eta \notin \mathcal{P}_{m}$, we get $\theta
_{1}\in \O(X,Y)$, and consequently, $\eta =\theta _{1}\theta _{2}\cdots
\theta _{l}\in \O(X,Y)$, a contradiction. So, we have verified that $%
\left\vert \mathcal{P}_{m}\right\vert <\left\vert B\right\vert $, i.e. the
relative rank of $\OP(X,Y)$ modulo $\O(X,Y)$ is greater than $\left(
\begin{array}{c}
n-1 \\
m%
\end{array}%
\right) $. But it is bounded by $1+\left(
\begin{array}{c}
n-1 \\
m%
\end{array}%
\right)$ due to Proposition \ref{le2}. This proves the assertion.
\end{proof}
~~\\

We finish this section with the characterization of the minimal relative
generating sets of $\OP(X,Y)$ modulo $\O(X,Y)$. We will
recognize that among them there are sets with size greater than $\rank(\OP(X,Y):\O(X,Y))$.

\begin{theorem}
Let $B\subseteq \OP(X,Y)$. Then $B$ is a minimal relative generating set of $\OP(X,Y)$ modulo $\O(X,Y)$ if and only if for the set $\widetilde{B}=\{\beta
\in B:\ker \beta \in \mathcal{Q}_{m}\} \subseteq B$ the following three statements are
satisfied:\newline
$(i)$ $\mathcal{P}_{m} \subseteq \{\ker \beta :\beta \in B\setminus \widetilde{B}\}$,\newline
$(ii)$ $|B\setminus \widetilde{B}| = |\mathcal{P}_{m}|$, \newline
$(iii)$ $\eta |_{Y}\in \left\langle \beta |_{Y}:\beta \in B\right\rangle $ but
$\eta |_{Y}\notin \left\langle \beta |_{Y}:\beta \in B\setminus \{\gamma
\}\right\rangle $ for any $\gamma \in \widetilde{B}$.
\end{theorem}

\begin{proof}
Suppose that the conditions $(i) - (iii)$ are satisfied for $\widetilde{B}=\{\beta \in B:\ker\beta \in \mathcal{Q}_{m}\}$.
We will show that $\mathcal{A} \subseteq \left\langle \O(X,Y),B \right\rangle$.
Let $\alpha \in \mathcal{A}\setminus\{\eta\}$.
Then there is a partition $P = \{A_{1},A_{2}<\cdots<A_{m}\}\in \mathcal{P}_{m}$ such that
$$\alpha = \alpha_{P}=\left(
\begin{array}{cccc}
A_{1} & A_{2} & \cdots  & A_{m} \\
a_{1} & a_{2} & \cdots  & a_{m}%
\end{array}%
\right), \mbox{ if } 1 \notin Y \mbox{ or } n \notin Y,$$
or
$$\alpha = \alpha_{P}=\left(
\begin{array}{ccccc}
A_{1} & A_{2} & \cdots  & A_{m-1} & A_m\\
a_{2} & a_{3} & \cdots  & a_{m} & a_1
\end{array}
\right), \mbox{ if } 1, n \in Y.$$
Notice that in the latter case $a_1 = 1$ and $a_m = n$.\\
Further, from $(i)$ it follows that there is $\beta \in B$ with $\ker\beta =\ker\alpha_P$, i.e. $\beta = \alpha_P$ or
$$\beta =\left(
\begin{array}{ccccccc}
A_{1} & A_{2} & \cdots  & A_{m-i+1} & A_{m-i+2} & \cdots  & A_{m} \\
a_{i} & a_{i+1} & \cdots  & a_{m} & a_{1} & \cdots  & a_{i-1}%
\end{array}%
\right)$$
for some $i\in \{3,\ldots,m\}$. It is easy to verify that $\alpha_{P}=\beta^{k}\in \left\langle
B\right\rangle$, for a suitable natural number $k$. Hence, $\{\alpha_{P}: P \in \mathcal{P}_{m}\} \subseteq \left\langle\O(X,Y),B \right\rangle$.
Further, $\ker \eta \in \mathcal{P}_{m}$, whenever $1\notin Y$ or $n\notin Y$,
and $\ker \eta \in \mathcal{Q}_{m}$ otherwise. Thus, there
is $\delta \in \left\langle \O(X,Y), B \right\rangle $ with $\ker \delta
=\ker \eta$. Then we obtain as above that $\eta =\delta^{l}\in \left\langle \O(X,Y),B \right\rangle$,
for a suitable natural number $l$. Consequently, $\left\langle
\O(X,Y),\mathcal{A}\right\rangle \subseteq \left\langle \O(X,Y),B \right\rangle$.
By Proposition \ref{le2}, we obtain $\OP(X,Y)=\left\langle \O(X,Y),B \right\rangle$.
The generating set $B$ is minimal by properties $(i)$ and $(ii)$
together with Lemma \ref{le3} and by the property $(iii)$ of $\widetilde{B}$.

Conversely, let $B$ be a minimal relative generating set of $\OP(X,Y)$ modulo
$\O(X,Y)$. By Lemma \ref{le3}, there is a set $\overline{B}\subseteq B$ such that
$\mathcal{P}_{m}=\{\ker \beta :\beta \in \overline{B}\}$ and $\left\vert
\overline{B}\right\vert =\left\vert \mathcal{P}_{m}\right\vert$.
Since $\OP(X,Y) = \left\langle \O(X,Y),B \right\rangle$, there are
$\beta _{1},\ldots ,\beta _{k}\in \O(X,Y)\cup B$ such that $\eta =\beta
_{1}\cdots \beta _{k}$. Without loss of generality, we can assume that
there is not $\gamma \in \{\beta _{i}:1\leq i\leq k$, $\ker \beta
_{i}\in \mathcal{Q}_{m}\}=: \widehat{B}$ such that $\eta$ is a product of transformations in
$\overline{B}\cup (\widehat{B}\setminus \{\gamma \})$. In the first
part of the proof, we have shown that $\overline{B}\cup \widehat{B}$ is a
relative generating set of $\OP(X,Y)$ modulo $\O(X,Y)$. Because of the
minimality of $B$, we have $B=\overline{B}\cup \widehat{B}$, where $\{\ker
\beta :\beta \in B\setminus \widehat{B}\}\supseteq\mathcal{P}_{m}$, $|B\setminus\widehat{B}| = |\overline{B}| = |\mathcal{P}_{m}|$ and $\eta |_{Y}\in
\left\langle \beta |_{Y}:\beta \in B\right\rangle$ but $\eta |_{Y}\notin
\left\langle \beta |_{Y}:\beta \in B\setminus \{\gamma \}\right\rangle $ for any $\gamma \in \widehat{B}$.
\end{proof}
~~\\

In particular, for the relative generating sets of minimal size we have

\begin{remark}
$B\subseteq \OP(X,Y)$ is a relative generating set of $\OP(X,Y)$ modulo $\O(X,Y)$
of minimal size if and only if $|\widetilde{B}| = 1$ if $1,n \in Y$ and $\widetilde{B} = \emptyset$, otherwise.
\end{remark}

\section{The relative rank of $\T(X,Y)$ modulo $\OP(X,Y)$}

In this section, we determine the relative rank of $\T(X,Y)$ modulo $\OP(X,Y)$
and characterize all minimal relative generating sets of $\T(X,Y)$ modulo $\OP(X,Y)$. Since $\O(X,Y)\leq \OP(X,Y)$, we see immediately that
$\rank(\T(X,Y):\OP(X,Y))\leq S(n,m)-\left(
\begin{array}{c}
n-1 \\
m-1%
\end{array}%
\right) +1$. First, we state a sufficient condition for a set $B\subseteq
\T(X,Y)$ to be a relative generating set of $\T(X,Y)$ modulo $\OP(X,Y)$.

\begin{lemma}\label{le7}
Let $B\subseteq \T(X,Y)$. If $\mathcal{R}_{m}\subseteq \{\ker \beta :\beta\in B\}$ and $\S(Y)\subseteq \left\langle \{\beta|_{Y}:\beta \in B\},\eta|_{Y}\right\rangle $ then $\left\langle \OP(X,Y),B\right\rangle =\T(X,Y)$.
\end{lemma}

\begin{proof}
Let $\gamma \in \T(X,Y)$ with $\rank \gamma =k \leq m$.
We will consider two cases.

\textit{Case 1.} Suppose that $\ker \gamma \in \mathcal{R}_{k}$.
Then $\ker \gamma$ contains a non-convex set which cannot be decomposed into two convex sets, which contain $1$ and $n$, respectively.
Since $k\leq m$, we can divide the partition $\ker \gamma$ into a partition
$P\in \mathcal{R}_{m}$ such that $P$ contains a non-convex set which cannot be decomposed into two convex sets, which contain $1$ and $n$, respectively
(if $k=m$ then we put $P=\ker \gamma$). Since $\mathcal{R}_{m}\subseteq \{\ker
\beta :\beta \in B\}$, there is $\lambda \in B$ with $\ker \lambda = P$.
It is clear that $X\lambda = Y$.\\
Further, let $X\gamma = \{y_1 < y_2 < \cdots < y_k\}$ and define the sets
\[A_i = \{x \in Y : x\lambda^{-1} \subseteq y_i\gamma^{-1}\}\]
for $i = 1,\ldots,k$. It is clear that $\{A_1, A_2, \ldots, A_k\}$ is a partition of $Y$.
Moreover, let $\{C_1 < C_2 < \cdots < C_k\} \in \mathcal{Q}_k$ be a partition of $X$
such that $|C_i \cap Y| = |A_i|$ for all $i = 1,\ldots,k$.
Let $A_i = \{a_{i_1} < a_{i_2} < \cdots < a_{i_{t_i}}\}$ and $C_i \cap Y = \{c_{i_1} < c_{i_2} < \cdots < c_{i_{t_i}}\}$ with $t_i \in \{1, \ldots, m\}$ for $i \in \{1, \ldots, k\}$.
We define a bijection
\[\sigma: \bigcup_{i=1}^k A_i = Y \longrightarrow \bigcup_{i=1}^k (C_i \cap Y) = Y\]
on $Y$ with $a_{i_l}\sigma = c_{i_l}$, for $l = 1,\ldots,t_i$ and $i = 1,\ldots,k$. Since $\sigma \in \S(Y)$ and
$\S(Y)\subseteq \left\langle \{\beta|_{Y}:\beta \in B\},\eta|_{Y}\right\rangle$ there is
$\mu \in \langle B, \eta\rangle$ with $\mu|_Y = \sigma$.\\
Finally, we define a transformation $\nu \in \O(X,Y) \subseteq \OP(X,Y)$ with
$\ker \nu = \{C_1 < C_2 < \cdots < C_k\}$ and $x\nu = y_i$ for all $x \in C_i$ and $i = 1,\ldots,k$.\\
Therefore, we have $\lambda, \mu, \nu \in \left\langle \OP(X,Y),B\right\rangle$ and it
remains to show that $\gamma = \lambda\mu\nu$, i.e. $\gamma \in \left\langle \OP(X,Y),B\right\rangle$.
Let $x \in X$. Then $x\gamma = y_i$ for some $i \in \{1,\ldots,k\}$ and we have
\[x\gamma = y_i \Rightarrow x\lambda = z \in A_i \Rightarrow z\mu = u \in C_i \cap Y \Rightarrow u\nu = y_i.\]
Hence, $x\gamma = y_i = x(\lambda\mu\nu)$ and we conclude $\gamma = \lambda\mu\nu$.

\textit{Case 2.} Suppose that $\ker \gamma \notin \mathcal{R}_{k}$, i.e.
$\ker \gamma \in \mathcal{Q}_{k}\cup \mathcal{P}_{k}$ and there is $\rho_{1}\in \OP(X,Y)$ with $\ker\rho_{1}=\ker \gamma$.
Further, there is a partition $P = \{D_{y} : y \in X\rho_{1}\} \in \mathcal{R}_{k}$
such that $y\in D_{y}$, for all $y\in X\rho_{1}$. Then we define a
transformation $\rho_{2} : X \rightarrow X\gamma$ with
$\ker\rho_{2} = P$ and $\{x\rho_2\} = y\rho_1^{-1}\gamma$ for all $x \in D_y$ and $y\in X\rho_{1}$.
Since $\ker \rho_{1}=\ker \gamma$, the transformation $\rho_2$ is
well defined and we have $\gamma = \rho_{1}\rho_{2}$.
Moreover, $\rho_2 \in \left\langle \OP(X,Y),B \right\rangle$
by Case 1 (since $\ker \rho_2 \in \mathcal{R}_{k}$) and thus
$\gamma =\rho_{1}\rho_{2}\in \left\langle \OP(X,Y),B \right\rangle$.
\end{proof}

\begin{lemma}\label{lenew}
$\left\langle \eta |_{Y}\right\rangle =\left\langle \left\{ \beta
|_{Y}:\beta \in \OP(X,Y)\right\} \right\rangle \cap \S(Y)$.
\end{lemma}

\begin{proof}
The inclusion $\left\langle \eta |_{Y}\right\rangle \subseteq \left\langle
\left\{ \beta |_{Y}:\beta \in \OP(X,Y)\right\} \right\rangle \cap \S(Y)$ is
obviously. Let now $\beta \in \OP(X,Y)$ with $\beta |_{Y}\in \S(Y)$. Then
there is $k\in \{1,\ldots ,m\}$ such that
\[
\beta =\left(
\begin{array}{cccccc}
A_{1} & \cdots  & A_{m-k+1} & A_{m-k} & \cdots  & A_{m} \\
a_{k} & \cdots  & a_{m} & a_{1} & \cdots  & a_{k-1}%
\end{array}%
\right)
\]
with $\{A_{1},A_{2}<\cdots <A_{m}\} \in \mathcal{P}_{m} \cup \mathcal{Q}_{m}$ and $a_{i}\in A_i$ for $i\in \{1,\ldots ,m\}$
since $Y$ is a transversal of $\ker \beta $. Thus,
\[
\beta |_{Y}=\left(
\begin{array}{cccccc}
a_{1} & \cdots  & a_{m-k+1} & a_{m-k} & \cdots  & a_{m} \\
a_{k} & \cdots  & a_{m} & a_{1} & \cdots  & a_{k-1}%
\end{array}%
\right) =(\eta |_{Y})^{m-k+1}\in \left\langle \eta |_{Y}\right\rangle.
\]
This shows that $\left\langle \left\{ \beta |_{Y}:\beta \in \OP(X,Y)\right\}
\right\rangle \cap \S(Y)\subseteq \{(\eta |_{Y})^{p}:p\in \mathbb{N}
\}=\left\langle \eta |_{Y}\right\rangle$.
\end{proof}

The following lemmas give us necessary conditions for a set $B\subseteq
\T(X,Y)$ to be a relative generating set of $\T(X,Y)$ modulo $\OP(X,Y)$.

\begin{lemma}\label{le8}
Let $B\subseteq \T(X,Y)\setminus \OP(X,Y)$ with $\left\langle \OP(X,Y),B\right\rangle =\T(X,Y)$.
Then $\S(Y)\subseteq \left\langle \{\beta |_{Y}:\beta \in B\},\eta|_{Y}\right\rangle$.
\end{lemma}

\begin{proof}
Let $\sigma \in \S(Y)$. We extend $\sigma$ to a transformation $\gamma : X \rightarrow Y$,
i.e. $\gamma|_{Y}=\sigma$. Hence, there are $\gamma_{1},\ldots,\gamma_{k}\in \OP(X,Y)\cup B$
(for a suitable natural number $k$) such that $\gamma=\gamma_{1}\cdots \gamma_{k}$.
Since the image of any transformation in $\T(X,Y)$ belongs to $Y$, we have
$\sigma =\gamma|_{Y}=\gamma_{1}|_{Y}\cdots \gamma_{k}|_{Y}$. Moreover, from $\sigma \in \S(Y)$,
we conclude $\gamma_{i}|_{Y}\in \S(Y)$ for $1\leq i\leq k$.
Let $\gamma_{i}\in \OP(X,Y)$ for some $i\in\{1,\ldots,k\}$. Then by Lemma \ref{lenew}
$$\gamma_{i}|_{Y}=\left(
\begin{array}{cccccc}
a_{1} & \cdots  & a_{t} & a_{t+1} & \cdots  & a_{m} \\
a_{m-t+1} & \cdots  & a_{m} & a_{1} & \cdots  & a_{m-t}%
\end{array}%
\right) \in \langle\eta |_{Y}\rangle$$
for a suitable natural number $t$. This shows $\sigma
\in \left\langle \{\beta |_{Y}:\beta \in B\},\eta |_{Y}\right\rangle$.
\end{proof}

\begin{lemma}\label{le9}
Let $B\subseteq \T(X,Y)\setminus \OP(X,Y)$ with $\left\langle \OP(X,Y),B\right\rangle =\T(X,Y)$.
Then $\mathcal{R}_{m}\subseteq \{\ker \beta :\beta \in B\}$.
\end{lemma}

\begin{proof}
Assume that there is $P\in \mathcal{R}_{m}$ with $P \not\in \{\ker \beta :\beta \in B\}$.
Let $\gamma \in \T(X,Y)$ with $\ker\gamma =P$. Then there are $\theta_{1}\in \OP(X,Y)\cup B$ and
$\theta_{2}\in \T(X,Y)$ such that $\gamma = \theta _{1}\theta _{2}$. Since
$\rank\gamma = m$, we obtain $\ker \gamma =\ker \theta_{1}= P$. Thus,
$\theta_{1}\not\in B$, i.e. $\theta_{1}\in \OP(X,Y)$ and
$\ker \theta_{1}\in \mathcal{Q}_{m}\cup \mathcal{P}_{m}$, contradicts $\ker \theta_{1} = P \in \mathcal{R}_{m}$.
\end{proof}
~~\\

Lemma \ref{le9} shows that $\rank(\T(X,Y):\OP(X,Y))\geq \left\vert \mathcal{R}_{m}\right\vert $. We will verify the equality.

\begin{lemma}\label{le10}
$\left\vert \mathcal{R}_{m}\right\vert = S(m,n)-\left(
\begin{array}{c}
n \\
m%
\end{array}%
\right)$.
\end{lemma}

\begin{proof}
The cardinality of the set $\mathcal{D}_{m}:=\mathcal{R}_{m}\cup \mathcal{P}_{m}$
was determined in \cite{TK}. The
authors show that $\left\vert \mathcal{D}_{m}\right\vert
=S(m,n)-\left(
\begin{array}{c}
n-1 \\
m-1%
\end{array}%
\right)$. Because of $\mathcal{R}_{m}\cap \mathcal{P}_{m}=\emptyset$, we
obtain $\mathcal{R}_{m}=\mathcal{D}_{m}\setminus \mathcal{P}_{m}$. Since
$|\mathcal{P}_{m}| =\left(
\begin{array}{c}
n-1 \\
m
\end{array}
\right)$ (see Remark \ref{re1}) it follows $\left\vert \mathcal{R}_{m}\right\vert =\left\vert \mathcal{D}_{m}\right\vert
-\left\vert \mathcal{P}_{m}\right\vert =S(m,n)-\left(
\begin{array}{c}
n-1 \\
m-1%
\end{array}%
\right) -\left(
\begin{array}{c}
n-1 \\
m%
\end{array}%
\right) =S(m,n)-\left(
\begin{array}{c}
n \\
m%
\end{array}%
\right)$.
\end{proof}
~~\\

Finally, we can state the relative rank of $\T(X,Y)$ modulo $\OP(X,Y)$.

\begin{theorem}\label{th11}
$\rank(\T(X,Y):\OP(X,Y))=S(m,n)-\left(
\begin{array}{c}
n \\
m%
\end{array}%
\right)$.
\end{theorem}

\begin{proof}
If $m=1$ then $\T(X,Y)=\OP(X,Y)$, i.e. $\rank(\T(X,Y):\OP(X,Y))=0$. On the other hand, we have
$S(1,n)=n=\left(
\begin{array}{c}
n \\
1%
\end{array}%
\right) $. Suppose now that $n\geq 2$.
By Lemmas \ref{le9} and \ref{le10}, we obtain $\rank(\T(X,Y):\OP(X,Y))\geq \left\vert
\mathcal{R}_{m}\right\vert =S(m,n)-\left(
\begin{array}{c}
n \\
m%
\end{array}%
\right) $. In order to prove the equality, we have to find a relative
generating set $B$ of $\T(X,Y)$ modulo $\OP(X,Y)$ with $\left\vert
B \right\vert = |\mathcal{R}_{m}|$.
We observe that for each $P \in \mathcal{R}_{m}$, there
is $\beta_P \in \T(X,Y)$ with $\ker \beta_P = P$, which will be fixed.
Let $\mathcal{B} := \{\beta_P : P \in \mathcal{R}_{m}\}$.
If $m=2$ then $\mathcal{R}_{m}=\emptyset$ and $\S(Y)=\{\eta |_{Y},(\eta |_{Y})^{2}\} = \langle \eta |_{Y} \rangle$, obviously.
If $m\geq 3$ then without loss of generality we can assume that there is $P'\in \mathcal{R}_{m}$ such
that $Y$ is a transversal of $\ker \beta_{P'}$ and
$\beta_{P'}|_{Y}=\left(
\begin{array}{ccccc}
a_{1} & a_{2} & a_{3} & \cdots  & a_{m} \\
a_{2} & a_{1} & a_{3} & \cdots  & a_{m}%
\end{array}%
\right)$.
It is well known that $\S(Y)=\left\langle \beta_{P'}|_Y, \eta|_Y \right\rangle$.
Hence, $\mathcal{B}$ is a relative generating set of $\T(X,Y)$ modulo $\OP(X,Y)$ by Lemma \ref{le7}.
Since $|\mathcal{B}| = |\mathcal{R}_{m}|$, we obtain the required result.
\end{proof}

Now we will characterize the minimal relative generating sets of $\T(X,Y)$ modulo $\OP(X,Y)$. The minimal relative generating sets do not coincide with the relative generating sets of size $\rank(\T(X,Y):\OP(X,Y))$.

\begin{theorem}
Let $B\subseteq \T(X,Y)$. Then $B$ is a minimal relative generating set of $\T(X,Y)$ modulo $\OP(X,Y)$ if and only if there is a set $\widetilde{B}\subseteq B$ such that the following three statements are satisfied:\newline
$(i)$ $\mathcal{R}_{m} \subseteq \{\ker \beta :\beta \in B\setminus \widetilde{B}\}$,\newline
$(ii)$ $|B\setminus \widetilde{B}| = |\mathcal{R}_{m}|$, \newline
$(iii)$ $\S(Y) \subseteq \left\langle \{\beta|_{Y}:\beta \in B\}, \eta|_{Y} \right\rangle$ but
$\S(Y) \nsubseteq \left\langle \{\beta |_{Y}:\beta \in B\setminus \{\gamma\}\}, \eta|_{Y} \right\rangle$
for any $\gamma \in B$ with $\ker \gamma \in \{\ker \beta : \beta \in \widetilde{B}\}$.
\end{theorem}

\begin{proof}
Suppose that the conditions $(i) - (iii)$ are satisfied. Then
by Lemma \ref{le7} we have $\left\langle \OP(X,Y),B\right\rangle =\T(X,Y)$.
It remains to show that $B$ is minimal. Assume that there is $\gamma
\in B$ such that $\left\langle \OP(X,Y),B\setminus \{\gamma \}\right\rangle
=\T(X,Y)$. Note that $\alpha \beta |_{Y}=\alpha |_{Y}\beta |_{Y}$ for all
$\alpha ,\beta \in \T(X,Y)$. Hence, we can conclude that $\S(Y)\subseteq
\left\langle \left\{ \beta |_{Y}:\beta \in \T(X,Y)\right\} \right\rangle
\subseteq \left\langle \left\{ \beta |_{Y}:\beta \in \OP(X,Y)\cup (B\setminus
\{\gamma \})\right\} \right\rangle =\left\langle \left\{ \beta |_{Y}:\beta
\in B\setminus \{\gamma \}\right\} ,\eta |_{Y}\right\rangle $ by Lemma \ref{lenew}.
Hence, $\ker \gamma \notin \{\ker\beta :\beta \in \widetilde{B}\}$ by $(iii)$.
This implies that $\gamma \in B\setminus \widetilde{B}$ and
$|(B\setminus \widetilde{B})\setminus \{\gamma \}|<|\mathcal{R}_{m}|$ by $(ii)$,
i.e. $\mathcal{R}_{m}\nsubseteqq \{\ker \beta :\beta \in (B\setminus \widetilde{B})\setminus \{\gamma \}\}$.
Since $\ker\gamma \notin \{\ker \beta :\beta \in \widetilde{B}\}$, we have
$\mathcal{R}_{m}\nsubseteqq \{\ker \beta :\beta \in (B\setminus \{\gamma \})\}$ and by
Lemma \ref{le9}, we obtain that $\left\langle \OP(X,Y),B\setminus \{\gamma
\}\right\rangle \neq \T(X,Y)$, a contradiction. This shows that $B$ is a
minimal relative generating set of $\T(X,Y)$ modulo $\OP(X,Y)$.

Conversely, let $B$ be a minimal relative generating set of $\T(X,Y)$ modulo
$\OP(X,Y)$. We have $\mathcal{R}_{m}\subseteq
\{\ker \beta :\beta \in B\}$ and $\S(Y)\subseteq \left\langle \{\beta
|_{Y}:\beta \in B\},\eta |_{Y}\right\rangle $ by Lemma \ref{le9} and Lemma \ref{le8},
respectively. Then there exists a set  $\widetilde{B}\subseteq B$ with
$|B\setminus \widetilde{B}| = |\mathcal{R}_{m}|$
and $\mathcal{R}_{m}\subseteq \{\ker \beta :\beta \in (B\setminus \widetilde{B})\}$.
For the set $\widetilde{B}$, the conditions $(i)$ and $(ii)$ are satisfied. Assume
now that there is $\gamma \in B$ with $\ker \gamma \in \{\ker \beta :\beta
\in \widetilde{B}\}$ such that $\S(Y)\subseteq \left\langle \{\beta
|_{Y}:\beta \in B\setminus \{\gamma\}\},\eta |_{Y}\right\rangle$. Then
because of $\mathcal{R}_{m}\subseteq \{\ker \beta :\beta \in (B\setminus \{\gamma\})\}$,
the set $B\setminus \{\gamma)$ is a relative generating set of
$\T(X,Y)$ modulo $\OP(X,Y)$ by Lemma \ref{le7}. This contradicts the minimality of $B$.
Consequently, $(iii)$ is satisfied.
\end{proof}
~~\\

In particular, for the relative generating sets of minimal size we have

\begin{remark}
$B\subseteq \T(X,Y)$ is a relative generating set of $\T(X,Y)$ modulo $\OP(X,Y)$ of minimal size if and only if
$\widetilde{B}=\emptyset$.
\end{remark}

\end{document}